\documentclass[a4paper]{article}
\usepackage{kdiss_eng}%

\usepackage[T1]{fontenc}
\usepackage{amsmath,graphicx,chicago}
\usepackage{amssymb}
\usepackage{multirow}
\usepackage{array}
\usepackage{hhline}
\usepackage[hidelinks]{hyperref}
\usepackage{kotex}
\usepackage{natbib}
\submit{final}
\thefirstpage{1}
\receive{2021}{8}{11}
\revise{2021}{9}{24}%��������
\accept{0000}{0}{0}%ä������
\volumn{32}{6}
\newcommand{\bea}{\begin{eqnarray*}}
\newcommand{\eea}{\end{eqnarray*}}
\heading
{Identifiability of Covariance Kernels in the Gaussian Process Regression Model}
{Jaehoan Kim $\cdot$ Jaeyong Lee}

\begin{document}

\title{
Identifiability of Covariance Kernels in the Gaussian Process Regression Model
}

\eauthor{
Jaehoan Kim\footnote{Undergraduate Student, Department of Mechanical Engineering, Seoul National University, Gwanak, Seoul 08826, Korea. E-mail: k1mjh6561@snu.ac.kr}$~\cdot$
Jaeyong Lee\footnote{Professor, Department of Statistics, Seoul National University, Gwanak, Seoul 08826, Korea. E-mail: jylc@snu.ac.kr}
}

%  Institutes after author name
\eaddress{
    ${}^{1}$Department of Mechanical Engineering, Seoul National University \\
    ${}^{2}$Department of Statistics, Seoul National University
}

% \eaccepted

\begin{eabstract}
Gaussian process regression (GPR) model is a popular nonparametric regression model.
In GPR, features of the regression function such as varying degrees of smoothness and periodicities are modeled through combining various covarinace kernels, which are supposed to model certain effects (\citealp{bousquet2011advanced}; \citealp{gelman2013bayesian}). The covariance kernels have unknown parameters which are estimated by the EM-algorithm or Markov Chain Monte Carlo. The estimated parameters are keys to the inference of the features of the regression functions, but identifiability of these parameters has not been investigated.
In this paper, we prove identifiability of covariance kernel parameters in two radial basis mixed kernel GPR and radial basis and periodic mixed kernel GPR. We also provide some examples about non-identifiable cases in such mixed kernel GPRs.
\end{eabstract}

\ekeywords{Covariance, Gaussian Process, Identifiability, Kernel, Regression.}

\section{Introduction}

% 가우스 과정 회귀모형의 중요성 - Importance of Gaussian process regression

%% Abstract에서 GPR에 대한 설명을 기입해두긴 했는데, Introduction에서도 다시 한 번 정의해주는 것이 더 좋을까 하여 우선 남겨두었습니다.
Gaussian process regression (GPR) is a widely used nonparametric regression method and can be used in time series analysis.
In the GPR modeling, seasonal effects are modeled with periodic covariance kernel and the smoothness of the regression function is with radial basis function (RBF) kernel. Various expansions including Gaussian process with colored noise (\citealp{li2020hybrid}) were developed.

% Identifiability의 문제를 생각하는 이유, 중요성 - Importance of Identifiability problem in Gaussian process regression
When observed data are modeled with Gaussian process regression, it is often insufficient to use just one type of covariance kernel. Most of observed features of regression function can be modeled by superposition of several effects such as seasonalities, smoothness, etc. For example, \citet{gelman2013bayesian} modeled weekly, monthly and yearly effects, and short-term and long-term trends by adding covariance kernels of corresponding effects. In \citet{wilson2013gaussian}, several kernels are used for prediction with patterns. In addition to Gaussian process, mixed models are widely in use because of its flexibility (\citealp{choi2017poisson}; \citealp{noh2018analysis}).

With the data and mixed model assumption, one can conclude the leverage of each effect with several algorithms mentioned below. For example, in \citet{gelman2013bayesian}, daily birthrate from 1969 to 1988 is analyzed with mixed GPR model, including weekend effect, seasonal trend, etc. The effect of each term is provided with graphs. Although this result is obtained under reasonable assumptions and algorithms, the uniqueness of such conclusion is assumed without proof. Here, the identifiability of mixed model becomes crucial. With the assurance of identifiability, at last, one can obtain the credibility of the analysis.

In this paper, we prove the identifiability of RBF-periodic mixed kernel GPR model and 2-RBF mixed kernel GPR model under certain conditions of observed data. The RBF and periodic kernels in the GPR model are used to model trends and periodicities in the data, respectively. The RBF-periodic mixed kernel GPR model is for the data with a general trend and a periodicity. The 2-RBF mixed kernel GPR model is used to model the long-term and short-term trends separately. We give numerical examples of unidentifiable cases with several number of observations.

% 이 문제를 다룬 과거의 논문들과 내용 - mixed kernel GPR을 이용한 model들이 이용되는 사례 및 identifiability problem 

Started from \citet{teicher1963identifiability}, previous identifiability problems are mostly defined and dealt in functional sense (\citealp{yakowitz1968identifiability}; \citealp{kostantinos2000gaussian}; \citealp{atienza2006new}; \citealp{barndorff1965identifiability}), which emphasize analytic aspect.

Practically, as in \citet{bousquet2011advanced} and \citet{gelman2013bayesian}, mixed kernel GPR model is widely in use. The methodologies of the estimation of parameters in mixture models include maximizing the posterior probability distribution function (\citealp{gelman2013bayesian}; \citealp{daemi2019identification}; \citealp{bousquet2011advanced}), Laplace approximation (\citealp{bousquet2011advanced}), kernel smoothing (\citealp{huang2014estimating}), and the usage of specific form of prior probability density function to improve practical identifiability (\citealp{gelman2013bayesian}).

% 본 논문으로 기대되는 효과
In addition to these strategies, this paper provide a firm theoretical basis to justify such method’s validity in practical uses. Although the results are limited to mixtures of two kernels, the methodology used in this paper can offer the basis of further research.

% 논문의 구성 설명 - Organization of this paper
This paper is organized as follows. In Section \ref{sec:pre}, we describe the identifiability problem of RBF-periodic and 2-RBF mixed kernel Gaussian process regression models and our results regarding the problem. In Section \ref{sec:proof}, we present theoretical proofs about two models, respectively. The numerical cases of non-identifiable model in RBF-periodic, and 2-RBF mixed kernel Gaussian process regression models are given in Section \ref{sec:num_ex}. Concluding remarks are given in Section \ref{sec:disc}.

\section{Identifiability of Mixed kernel Gaussian Process Regression Model} \label{sec:pre}

\subsection{Notation}

Suppose we observe $n \in \mathbb{N}$ samples with $d$-dimensional covariate and 1-dimensional response ($x_1, y_1$), ..., ($x_n, y_n$), where $\mathbb{N}$ is the set of natural numbers. The observations follow Gaussian process regression (GPR) model:
\bea
    y_i = f(x_i) + \epsilon_i,\; \epsilon_i \stackrel{iid}{\sim} N(0, {\sigma_{0}}^2), i=1,\cdots,n,
\eea
where $\sigma_{0}^2>0$ is the observational variance, $f(x)$ follows Gaussian process with mean function $\mu(x)$ and covariance kernel $K(x, x')$, and we denote $f(x) \sim GP(\mu, K)$. Note that under the GPR model assumption
\bea
(f(x_1), \cdots, f(x_n))^T \sim N((\mu(x_1), \mu(x_2), ..., \mu(x_n))^T,\; K(x_1, x_2, \cdots, x_n)),
\eea
where $ K(x_1, \cdots, x_n) = [K(x_i, x_j)]$ represents an $ n \times n $ covariance matrix, called a kernel matrix, or Gram matrix.
In this paper, we consider two covariance kernel functions, the radial basis function kernel and the periodic kernel.

The RBF kernel is defined as
\bea
K_{RBF}(x, x'; \theta = (\sigma,l)) = \sigma^2 \times \exp(\frac{-||x-x'||^2}{l^2}), \; x, x' \in \mathbb{R}^d, \sigma, l>0.
\eea
We denote $\theta=(\sigma, l)$, where $\sigma$ is the parameter representing the impact of RBF kernel, and $l$ the smoothness of RBF kernel.  The periodic kernel is defined as
\bea
K_{period}(x, x'; \eta = (\tau, s)) = \tau^2 \times \exp(\frac{-2\sin^2({\pi(x-x')/p})}{s^2}), \quad x, x' \in \mathbb{R}, \tau, s, p>0,
\eea
where $\tau$ denotes the impact of the periodic kernel, $s$ the smoothness and $p$ the period of the model. Note that the periodic kernel is defined on $x \in \mathbb{R}$ while the RBF kernel is on $x \in \mathbb{R}^d$. If $f \sim GP(0, K_{period})$ with period $p$ and $|x-x'| = kp$ with an integer $k$, $corr(f(x), f(x')) = 1$. In this paper, we will consider the period $p$ as a known constant.

Consider a mixed kernel GPR model with two covariance kernel functions: 
\begin{equation}
    y = \mu(x) + f_1(x) + f_2(x) + \epsilon,
\end{equation}
where $\mu$ is a mean function, $f_1 \sim GP(0, K_1)$, $f_2 \sim GP(0, K_2)$ and $\epsilon \sim N(0, {\sigma_0}^2)$. In this case, $f_1+f_2 \sim GP(0, K_1+K_2)$. For the simplicity of discussion, we assume $\mu(x) = 0, \forall x$ for the rest of the paper.

\subsection{Identifiability of mixed kernel Gaussian process regression model}

The kernel matrices of the RBF kernel and the periodic kernel are written as
\bea
G_{RBF}=G_{RBF}(\theta) &=& [K_{RBF}(x_i, x_j;\theta = (\sigma, l))]\\
G_{period}=G_{period}(\eta) &=& [K_{period}(x_i, x_j;\eta = (\tau, s))]. 
\eea

Consider a RBF-periodic mixed kernel GPR model.
\begin{equation}\label{model:mixed}
y = f_1(x) + f_2(x) + \epsilon
\end{equation}
where
$$
f_1(x) \sim GP(0, G_{RBF}(\theta)),\quad f_2(x) \sim GP(0, G_{period}(\eta)),\quad \epsilon \sim N(0, \sigma_0^2).
$$
Due to the affine property of Gaussian distribution, 
\bea
f_1(x) + f_2(x) \sim GP(0, G_{RBF}(\theta)+G_{period}(\eta))
\eea
holds. If $\theta$ and $\eta$ can be uniquely determined from $G_{RBF}+G_{period}$, the model is identifiable.

\begin{theorem}[Identifiability of RBF-Periodic mixed kernel GPR model]\label{thm:RBF-Periodic}

Let
$X = \{ |x_i-x_j|$ : $1\le i, j \le n\}$. For the RBF-Periodic mixed GPR model with period $p$, if there exist $\alpha, \beta \in X$ such that $\displaystyle\frac{\alpha}{p} \in \mathbb{N}$ and $\displaystyle\frac{\beta}{p} \notin \mathbb{N}$, the RBF-periodic mixed kernel GPR model is identifiable.
\end{theorem}

This theorem means that when 
\bea
G_{RBF}(\theta_1)+G_{period}(\eta_1) = G_{RBF}(\theta_2)+G_{period}(\eta_2)
\eea
holds,
\bea
\theta_1 = \theta_2,\quad \eta_1 = \eta_2
\eea
is satisfied.

Similar to RBF-periodic mixed kenrel GPR model, suppose 2-RBF mixed kernel GPR model using (\ref{model:mixed}) with
\bea
f_1(x) \sim GP(0, G_{RBF}(\theta_1)),\; f_2(x) \sim GP(0, G_{RBF}(\theta_2)),\; \epsilon \sim N(0, \sigma_0^2).
\eea
To distinguish two RBF kernel matrices in the model, let $l_1 < l_2$ for $\theta_1 = (\sigma_1, l_1),\;\theta_2 = (\sigma_2, l_2)$. 

\begin{theorem}[Identifiability of 2-RBF mixed model]\label{thm:2-RBF}
Let $X = \{\lVert x_i-x_j \rVert$ : $1\le i, j \le n\}$. If $|X| \ge 4$, the 2-RBF mixed kernel GPR model is identifiable, where $|X|$ is the cardinality of the set X.
\end{theorem}
This theorem means that for $X$ defined above, if $|X| \ge 4$, 
\bea
G_{RBF}(\theta_{11})+G_{RBF}(\theta_{12}) = G_{RBF}(\theta_{21})+G_{RBF}(\theta_{22})
\eea
and $l_{11}<l_{12}$ and  $l_{21}<l_{22}$,
\bea
\theta_{11} = \theta_{12},\quad \theta_{21} = \theta_{22}
\eea
is satisfied.
The proof of Theorem 1 and 2 are given in the next section.

%%%%%%%%%%%%%%%%%%%%%%%%%%%%%%%%%%%%%%%%%%%%%%%%%%%%%%%%%%%%%%%%%%%%%%%%%%%%%%%%%%%%%%%%%%%%%%%%%%%%%%%%%%%%%%

\section{Proofs}\label{sec:proof}
\subsection{Proof of Theorem \ref{thm:RBF-Periodic}}

Suppose $X$ satisfies condition mentioned in \ref{thm:RBF-Periodic}. Suppose that $G_{RBF}$ is kernel matrix of RBF kernel and $G_{period}$ of periodic kernel, with period $p$. In this model, let $\theta_i = (\sigma_i, l_i)$ and $\eta_i = (\tau_i, s_i)$\;$(i = 1, 2)$.

Suppose
\begin{equation}\label{eqn:RBF-Periodic mixed}
G_{RBF}(\theta_1) + G_{period}(\eta_1) = G_{RBF}(\theta_2) + G_{period}(\eta_2),\quad 
\forall \theta_1, \theta_2, \eta_1, \eta_2 
\end{equation}
which is equivalent to
\begin{equation}\label{eqn:RBF-Periodic mixed_element_ver}
    K_{RBF}(x_i, x_j;\theta_1) + K_{period}(x_i, x_j;\eta_1) = K_{RBF}(x_i, x_j;\theta_2)+ K_{period}(x_i, x_j;\eta_2),
\end{equation}
where $i, j=1,\cdots,n.$
Note that (\ref{eqn:RBF-Periodic mixed_element_ver}) is equivalent to
\begin{equation}\label{eqn:linear relation}
    (\sigma_1^2, \tau_1^2 \exp(-\frac{1}{s_{1}^2}), -\sigma_2^2, -\tau_2^2\exp(-\frac{1}{s_{2}^2}))^T \cdot v(x) = 0, \forall x \in X, 
\end{equation}
where
\bea
v(x) = (\exp(-\frac{x^2}{l_1^2}), \exp(\frac{\cos(2\pi x/p)}{s_1^2}),\exp(-\frac{x^2}{l_2^2}),\exp(\frac{\cos(2\pi x/p)}{s_2^2}))^T.
\eea
This is because for $x=|x_i-x_j|$, (\ref{eqn:RBF-Periodic mixed_element_ver}) is written as 
\bea
\sigma_1^2\exp(-\frac{x^2}{l_1^2})+\tau_1^2\exp(-\frac{cos(2\pi x/p)-1}{s_1^2}) = \sigma_2^2\exp(-\frac{x^2}{l_2^2})+\tau_2^2\exp(-\frac{cos(2\pi x/p)-1}{s_2^2}),
\eea
which is equivalent to (\ref{eqn:linear relation}).

Therefore, when we define $V = \{v(x)|x \in X\}$ and $\mathcal{C}(V)$ as a set of linear combinations of $V$, ker\;$\mathcal{C}(V)$ includes nonzero vector, and rank\;$\mathcal{C}(V) \;\le\; 3$. 

We first show that if $l_1 \neq l_2$ and $s_1 \neq s_2$, rank\;$\mathcal{C}(V)$\; = 4, and that if $l_1 = l_2$ and $s_1 \neq s_2$ or $l_1 \neq l_2$ and $s_1 = s_2$, a contradiction occurs. Finally, we finish the proof by showing that if $l_1 = l_2$ and $s_1 = s_2$, then $\sigma_1=\sigma_2$ and $\tau_1 = \tau_2$, which implies that the model is identifiable. For the proof, we need following lemmas.

\begin{lemma}\label{lemma:condition of x}
If there exists $\alpha, \beta \in X$ such that
\begin{equation}\label{eqn:condition of x}
\frac{\alpha}{p}\in \mathbb{N},\; 
\frac{\beta}{p} \notin \mathbb{N},
\end{equation}
then there exists $m \in \mathbb{N}$ and $q>0$ such that 
\begin{equation}\label{eqn:condition of lemma3.1}
\frac{q}{p}\notin \mathbb{N},\;  \{0, mp, q, mp+q\} \subset X \; or\; \{0, q, mp-q, mp\} \subset X.
\end{equation}
\end{lemma}
\begin{proof}
Without loss of generality, suppose that 
\bea 
X=\{x_1<x_2<...<x_n\}
\eea
and let 
\bea
a_i = x_{i+1} - x_{i}, \;i = 1,\cdots, n-1.
\eea

Case I. Suppose there exists $a_i$ such that $\displaystyle\frac{a_i}{p} \in \mathbb{N}$. Let 
\bea
t = \min\{i:1\le i \le n-1,\;\frac{a_i}{p} \in \mathbb{N}\}.
\eea

Suppose $t \ge 2$. By the definition of $t$,  $\displaystyle \frac{a_{t-1}}{p} \notin \mathbb{N}$.\;Let $q = a_{t-1}$ and $\displaystyle m = \frac{a_t}{p}$. Then
\bea
x_t-x_t = 0,\; x_t-x_{t-1} = q,\; x_{t+1} - x_t = a_t = mp, \text{ and } x_{t+1} - x_{t-1} = mp+q.
\eea
Thus, $\{0, mp, q, mp+q\} \subset X$.

Suppose $t = 1$. Let
\bea
s = \min\{i:1\le i \le n-1,\;\frac{a_i}{p} \notin \mathbb{N}\}
\eea
Note that the existence of s follows from (\ref{eqn:condition of x}). Since $t=1$, $s \ge 2$, and thus $\displaystyle \frac{a_{s-1}}{p} \in \mathbb{N}$ and $\displaystyle \frac{a_{s}}{p} \notin \mathbb{N}$. Define $q = a_s$, and $m = \displaystyle   \frac{a_{s-1}}{p}$.\; Then,
\bea
x_s-x_s = 0,\; x_{s+1} - x_s = a_s = q,\; x_s-x_{s-1}= mp,\; \text{and } x_{s+1} - x_{s-1} = mp+q.
\eea
Therefore, $\{0, mp, q, mp+q\} \subset V$.

Case II. Suppose that for all $1 \le i \le n$, $\displaystyle \frac{a_i}{p} \notin \mathbb{N}$.

Suppose there exists $(i, j)$ such that $\displaystyle \frac{x_j-x_i}{p}\in \mathbb{N}$ and $j \le n-1$.
For this $i, j$, define $\displaystyle m = \frac{x_j-x_i}{p}$ and set $q=a_j$. Then 
\bea
x_j-x_j=0,\; x_j-x_i=mp,\; x_{j+1}-x_j=q,\text{ and } x_{j+1}-x_i=mp+q,
\eea
which satisfies (\ref{eqn:condition of lemma3.1}).

Suppose there does not exist $(i, j)$ such that $\displaystyle \frac{x_j-x_i}{p} \in \mathbb{N}$ with $j \le n-1$. This implies that for all $(i, j)$ which satisfies $\displaystyle \frac{x_j-x_i}{p} \in \mathbb{N}, \; j=n$. Furthermore, we can see that there should be only one pair of (i, j). If it is not, for $i_1$ and $i_2$ such that $\displaystyle \frac{x_n-x_{i_1}}{p} \in \mathbb{N}$ and $\displaystyle \frac{x_n-x_{i_2}}{p} \in \mathbb{N},\; \displaystyle \frac{|x_{i_1}-x_{i_2}|}{p} \in \mathbb{N}$ holds, which contradicts the assumption. Therefore, since $\displaystyle\frac{a_{n-1}}{p} \notin \mathbb{N}$, $i$ which satisfies $\displaystyle\frac{x_n-x_i}{p} \in \mathbb{N}$ should satisfy $i \le n-2$. Define 
\bea
\displaystyle m = \frac{x_n-x_i}{p}, \text{ and } q = x_n-x_{n-1}.
\eea
Then
\bea
x_i-x_i=0,\; x_{n-1}-x_i = mp-q,\; x_n-x_{n-1} = q,\text{ and } x_n-x_i = mp,
\eea
which means that $\{0, mp-q, q, mp\} \subset X$, where $0<q<mp$.  
This concludes the proof of Lemma \ref{lemma:condition of x}.
\end{proof}

\begin{lemma}\label{lemma:2 exp vectors}
If $ k \neq l$ and $a\neq b$, two vectors $(e^{ka}, e^{kb})$ and $(e^{la}, e^{lb})$ are linearly independent.
\end{lemma}
\begin{proof}
The conclusion of the lemma is equivalent to 
$\begin{vmatrix}e^{ka} & e^{la} \\ e^{kb} & e^{lb}\end{vmatrix} \neq 0$ 
when
$ k \neq l$ and $a \neq b$. This is trivial because $\displaystyle\begin{vmatrix}e^{ka} & e^{la} \\ e^{kb} & e^{lb}\end{vmatrix} = e^{ka+lb} \times (1-e^{(k-l)(a-b)}).\; $
\end{proof}

\begin{lemma}\label{lemma:monotone function}
For $t, s>0$, $\displaystyle I(t, s) = \frac{te^{-ts}}{1-e^{-ts}}$ is monotonely decreasing in t.
\end{lemma}
\begin{proof}
Note $\displaystyle \frac{\partial I}{\partial t} = e^{-ts}\times \frac{1-ts-e^{-ts}}{(1-e^{-ts})^2}$. When we set 
$r(x) = 1-x-e^{-x}$, $r'(x) = -1+e^{-x}$, which implies that
$r$ is monotonely decreasing on $\{x>0\}$. Therefore, for all $\displaystyle t, s>0, \frac{\partial I}{\partial t} = \frac{e^{-ts}}{(1-e^{-ts})^2} \times r(ts) < 0$, and $I(t, s)$ is monotonely decreasing in t.
\end{proof}

\begin{lemma}\label{lemma: injectiive property}
For $p, l>0$, let $\displaystyle h(x;p, l) = \frac{\exp(-(x+p)^2/l^2)-1}{\exp(-p^2/l^2)-1}$. If there exists $x>0$ such that $h(x;p, l_1) = h(x;p, l_2)$ for $l_1, l_2>0$, then $l_1 = l_2$.
\end{lemma}

\begin{proof}
For $x, y, p>0$, define 
\bea
\displaystyle s(x, y;p) = \frac{\exp(-(x+p)^2y^2)-1}{\exp(-p^2y^2)-1}.
\eea
Then it is equivalent to show that if there exists $x>0$ such that $s(x, 1/l_1; p) = s(x, 1/l_2; p)$, $l_1 = l_2$.
\bea
    \displaystyle
\frac{\partial s}{\partial y} = \frac{2y(1-\exp(-p^2y^2))(1-\exp(-(x+p)^2y^2))}{(\exp(-p^2y^2)-1)^2} \times \\
(\frac{(x+p)^2\exp(-(x+p)^2y^2)}{1-\exp(-(x+p)^2y^2)} - \frac{p^2\exp(-p^2y^2)}{1-\exp(-p^2y^2)}).
\eea
When we define 
\bea
\displaystyle
g(x, y;p) &=& \frac{2y(1-\exp(-p^2y^2))(1-\exp(-(x+p)^2y^2))}{(\exp(-p^2y^2)-1)^2},\\
\eea
the following holds:
\bea
\displaystyle
\frac{\partial s}{\partial y} &=& g(x, y;p)(I((x+p)^2, y^2) - I(p^2, y^2)).
\eea
When $x, y, p>0, g(x, y;p) >0$ and $(x+p)^2>p^2>0$ holds. By Lemma \ref{lemma:monotone function}, $\displaystyle \frac{\partial s}{\partial y} < 0$. Therefore, if $\displaystyle s(x, \frac{1}{l_1};p) = s(x, \frac{1}{l_2};p)$, $l_1 = l_2$ holds. This concludes the proof.
\end{proof}

\begin{lemma}\label{lemma: mainstream}
Let
\bea
v(x) = {(\exp(-\frac{x^2}{l_1^2}), 
\exp(-\frac{x^2}{l_2^2}),
\exp(\frac{\cos(2\pi x/p)}{s_1^2}),
\exp(\frac{\cos(2\pi x/p)}{s_2^2}))}^T
\eea
and $V = \{v(x)|x\in X\}$. If $l_1 \neq l_2, s_1 \neq s_2,$ and
\bea
\{0, mp, q, mp+q\}\;or\;\{0, q, mp-q, mp\} \subset X,\,
\eea
then rank $\mathcal{C}(V) = 4$. 
\end{lemma}

\begin{proof}
First, suppose that $(0, mp, q, mp+q) \in X$. Then
\bea
\displaystyle
v(0) = 
\left(
\begin{matrix}
1\\
1\\
\displaystyle\exp(\frac{1}{s_1^2})\\
\displaystyle\exp(\frac{1}{s_2^2})
\end{matrix}
\right)&,& 
v(mp) = 
\left(
\begin{matrix}
\displaystyle\exp(-\frac{m^2p^2}{l_1^2})\\
\displaystyle\exp(-\frac{-m^2p^2}{l_2^2})\\
\displaystyle\exp(\frac{1}{s_1^2})\\
\displaystyle\exp(\frac{1}{s_2^2})
\end{matrix}
\right), \\
v(q) = 
\left(
\begin{matrix} 
\displaystyle\exp(-\frac{q^2}{l_1^2})\\
\displaystyle\exp(-\frac{q^2}{l_2^2})\\
\displaystyle\exp(\frac{\cos(2\pi q/p)}{s_1^2})\\
\displaystyle\exp(\frac{\cos(2\pi q/p)}{s_2^2})
\end{matrix} 
\right)&,&
v(mp+q) = 
\left(
\begin{matrix}
\displaystyle\exp(-\frac{(mp+q)^2}{l_1^2})\\
\displaystyle\exp(-\frac{-(mp+q)^2}{l_2^2})\\
\displaystyle\exp(\frac{\cos(2\pi q/p)}{s_1^2})\\
\displaystyle\exp(\frac{\cos(2\pi q/p)}{s_2^2})
\end{matrix} 
\right). 
\eea
Define
\bea
w_1 = v(mp)&-&v(0)=\left(\begin{matrix} 
\displaystyle\exp(-\frac{m^2p^2}{l_1^2})-1 \\ 
\displaystyle\exp(-\frac{m^2p^2}{l_2^2})-1\\
0\\0
\end{matrix} \right)\\ 
w_2 = v(mp+q)&-&v(q)=\left(\begin{matrix} \displaystyle\exp(-\frac{(mp+q)^2}{l_1^2})-\exp(-\frac{q^2}{l_1^2}) \\
\displaystyle\exp(-\frac{(mp+q)^2}{l_2^2})-\exp(-\frac{q^2}{l_2^2})\\
0\\0
\end{matrix} \right).
\eea
Since 
\bea
\mathcal{C}(\{v(0), v(mp), v(q), v(mp+q)\}) = 
\mathcal{C}(\{v(0), v(q), w_1, w_2\}),
\eea
it suffices to show that rank $\mathcal{C}(\{v(0), v(q), w_1, w_2\}) = 4$ or equivalently that $\{v(0), v(q), w_1, w_2\}$ is a linearly independent set. Suppose that for $a, b, c, d \in \mathbb{R}$,  $av(0) + bv(q) + cw_1 + dw_2 = 0$. 

By comparing the $3^{rd}$, $4^{th}$ components, we can see that
\bea
a
\left(
\begin{matrix}
\displaystyle\exp(\frac{1}{s_1^2})\\ 
\displaystyle\exp(\frac{1}{s_1^2})
\end{matrix}
\right)
+
b
\left(
\begin{matrix}
\displaystyle\exp(\frac{\cos(2\pi q/p)}{s_1^2})\\ 
\displaystyle\exp(\frac{\cos(2\pi q/p)}{s_2^2})
\end{matrix}
\right) 
=
\left(
\begin{matrix}
0\\0 
\end{matrix}
\right)
\eea
should be satisfied. By Lemma \ref{lemma:2 exp vectors}, the two 2-dimensional vectors are linearly independent, which implies that $a$ and $b$ are 0.

By comparing the $1^{st}$ and $2^{nd}$ components, we have
\bea
c
\left(
\begin{matrix}
     \displaystyle\exp(-\frac{m^2p^2}{l_1^2})-1\\
     \displaystyle\exp(-\frac{m^2p^2}{l_2^2})-1
\end{matrix} 
\right) 
+ 
d
\left(
\begin{matrix}
     \displaystyle\exp(-\frac{(mp+q)^2}{l_1^2})-\exp(-\frac{q^2}{l_1^2})\\
     \displaystyle\exp(-\frac{(mp+q)^2}{l_2^2})-\exp(-\frac{q^2}{l_2^2})
\end{matrix} 
\right) 
=
\left(
\begin{matrix}
     0\\
     0
\end{matrix}
\right).
\eea
This is equivalent to

\bea
B
\left(
\begin{matrix}
     c\\d
\end{matrix}
\right) 
= 
\left( 
\begin{matrix}
    0\\0
\end{matrix}
\right),
\eea
where
\bea
B = \left(
\begin{matrix}
\displaystyle\exp(-\frac{m^2p^2}{l_1^2})-1, & \displaystyle\exp(-\frac{(mp+q)^2}{l_1^2})-\exp(-\frac{q^2}{l_1^2})\\
\displaystyle\exp(-\frac{m^2p^2}{l_2^2})-1, & \displaystyle\exp(-\frac{(mp+q)^2}{l_2^2})-\exp(-\frac{q^2}{l_2^2})
\end{matrix}
\right).
\eea
Suppose that $(c, d)\neq(0, 0)$, which means $\det B = 0$. When we define 
\bea
\displaystyle g(x;p, l) = \exp(-\frac{(x+p)^2}{l^2}) - \exp(\frac{x^2}{l^2}),
\eea
\bea
\det B = g(0;mp, l_1)g(q;mp, l_2) - g(0;mp, l_2)g(q, mp, l_1)
\eea
holds. When $x,\;p>0,\;g(x;p, l) \neq 0$ holds. Therefore, 
$\det B = 0$ is equivalent to 
\bea
\displaystyle\frac{g(q;mp, l_1)}{g(0;mp, l_1)}= \frac{g(q;mp, l_2)}{g(0;mp, l_2)}.
\eea
Define 
\bea
h(x;p, l) = \frac{\exp(-(x+p)^2/l^2) - \exp(-x^2/l^2)}{\exp(-p^2/l^2) - 1}.
\eea
Note
\bea
\displaystyle h(q;mp, l) = \frac{g(q;mp, l)}{g(0;mp, l)}.
\eea
Consequently, $\det B = 0$ is equivalent to 
\bea
h(q;mp, l_1) = h(q;mp, l_2).
\eea
By Lemma \ref{lemma: injectiive property}, it is equivalent to $l_1 = l_2$, which contradicts to the fact that $l_1 \neq l_2$. Therefore, the assumption that $\det B \neq 0$ cannot be satisfied, which implies that $(c, d)$ should be $(0, 0)$. In conclusion, if
\bea
av(0) + bv(q) + cw_1 + dw_2 = 0
\eea
holds,
\bea
(a, b, c, d) = (0, 0, 0, 0).
\eea
Accordingly, $\{v(0), v(q), w_1, w_2\}$ is a set of linearly independent vectors, which implies rank $\mathcal{C}(V)=4$.

Now, suppose that $(0, q, mp-q, mp) \in X$. In this case, by the definition of X, $0<q<mp$. We can prove the result of the lemma similarly with
\bea
B = 
\left( 
\begin{matrix}
\displaystyle\exp(-\frac{m^2p^2}{l_1^2})-1, & \displaystyle\exp(-\frac{(mp-q)^2}{l_1^2})-\exp(-\frac{q^2}{l_1^2})\\
\displaystyle\exp(-\frac{m^2p^2}{l_2^2})-1, & \displaystyle\exp(-\frac{(mp-q)^2}{l_2^2})-\exp(-\frac{q^2}{l_2^2})
\end{matrix}
\right).
\eea
Therefore, $\det B = 0$ is equivalent to
\bea
h(-q;mp, l_1) = h(-q;mp, l_2).
\eea
By the definition of $h$,
\bea
\displaystyle\frac{1}{h(-x;p, l)}= h(x;p-x,l) 
\eea
holds. Consequently,
\bea
h(-q;mp,l_1) = h(-q;mp,l_2)
\eea
is equivalent to
\begin{equation}\label{eqn:h is injective}
h(q;mp-q, l_1) = h(q;mp-q, l_2). 
\end{equation}
Since $0<q<mp$, according to Lemma \ref{lemma: injectiive property}, $l_1 = l_2$ holds. This contradicts the assumption of Lemma \ref{lemma: mainstream}. Therefore, $\{v(0), v(q), v(mp-q), v(mp)\}$ is a linearly independent set, which means that rank $\mathcal{C}(\{v(0), v(q), v(mp-q), v(mp)\}) = 4$.
\end{proof}

Therefore, if $G_{RBF}(\theta_1) + G_{period}(\eta_1) = G_{RBF}(\theta_2) + G_{period}(\eta_2)$ holds, $l_1 \neq l_2$ and $s_1 \neq s_2$ cannot be satisfied simultaneously. Now we consider two more cases to conclude the proof.

Case I. Suppose $l_1=l_2$ and $s_1 \neq s_2$. This implies that
\begin{equation}\label{eqn: relation in case 1}
\displaystyle(\sigma_1^2 - \sigma_2^2)\exp(-\frac{x^2}{l^2}) + \tau_1^2e^{-1/s_1^2}\exp(\frac{\cos(2\pi x/p)}{s_1^2}) = \tau_2^2e^{-1/s_2^2}\exp(\frac{\cos(2\pi x/p)}{s_2^2}).
\end{equation}
Define 
\bea
w(x) = (\exp(-\frac{x^2}{l_1^2}), \exp(\frac{\cos(2\pi x/p)}{s_1^2}),\exp(\frac{\cos(2\pi x/p)}{s_2^2}))^T, \; W = \{w(x)|x \in X\}.
\eea 
To satisfy (\ref{eqn: relation in case 1}), rank $\mathcal{C}(W) \le 2$ should be satisfied. However, since
\bea
\mathcal{C}(\{w(0), w(mp)\})=\mathcal{C}(\{(1, 0, 0)^T, (0, \exp(\frac{1}{s_1^2}), \exp(\frac{1}{s_2^2}))^T\})
\eea
holds and by Lemma \ref{lemma:2 exp vectors},
\bea
w(q) = (\exp(-\frac{q^2}{l_1^2}),\exp(\frac{\cos(2\pi q/p)}{s_1^2}) , \exp(\frac{\cos(2\pi q/p)}{s_2^2}))^T
\eea
cannot be expressed as a linear combination of $(1, 0, 0)^T$ and $(0, \exp(1/s_1^2), \exp(1/s_2^2))^T$. Therefore, rank $\mathcal{C}(W) = 3$ and $\tau_1=\tau_2=0$ holds, which contradicts the initial assumption.

Case II. Suppose $l_1 \neq l_2$ and $s_1 = s_2$. Then
\begin{equation}\label{eqn:relation in case 2}
    \sigma_1^2\exp(-\frac{x^2}{l_1^2}) + (\tau_1^2-\tau_2^2)e^{-1/s_1^2}\exp(\frac{\cos(2\pi x/p)}{s_1^2}) = \sigma_2^2\exp(-\frac{x^2}{l_2^2})
\end{equation}
holds. Define
\bea
u(x) = (\exp(-\frac{x^2}{l_1^2}), \exp(-\frac{x^2}{l_2^2}),\exp(\frac{\cos(2\pi x/p)}{s_1^2}))^T, \; U = \{u(x)|x \in X\}.
\eea 
To satisfy (\ref{eqn:relation in case 2}) with $\sigma_1\sigma_2 \neq 0$, rank $\mathcal{C}(U) \le 2$ should be satisfied. Since
\bea
\displaystyle
u(0) = (1, 1, \exp(\frac{1}{s_1^2}))^T, u(mp) = (\exp(-\frac{m^2p^2}{l_1^2}), \exp(-\frac{m^2p^2}{l_2^2}), \exp(\frac{1}{s_1^2}))^T
\eea
are linearly independent vectors by Lemma \ref{lemma:2 exp vectors}, $\mathcal{C}(U) \;\subset\; \mathcal{C}(\{u(0), u(mp)\})$ should be satisfied to make rank $\mathcal{C}(U) \le 2$. To satisfy this, $u(mp+q) - u(q)$ or $u(mp-q)-u(q)$ should be included in $ \mathcal{C}(\{u(0), u(mp)\})$. For $u(mp+q) - u(q)$, it is equivalent to
\begin{equation}\label{eqn:rank3prove}
    u(mp+q) - u(q) = cu(0) + du(mp)
\end{equation}
for some $c$ and $d$. By comparing $3^{rd}$ element of (\ref{eqn:rank3prove}), $c+d = 0$ should hold. Then (\ref{eqn:rank3prove}) is equivalent to
\bea
\left(
\begin{matrix}
\displaystyle\exp(-\frac{(mp+q)^2}{l_1^2}) - \exp(-\frac{q^2}{l_1^2}) \\ 
\displaystyle\exp(-\frac{(mp+q)^2}{l_2^2}) - \exp(-\frac{q^2}{l_2^2})
\end{matrix}
\right) = d 
\left(
\begin{matrix}
\displaystyle\exp(-\frac{m^2p^2}{l_1^2})-1\\ 
\displaystyle\exp(-\frac{m^2p^2}{l_2^2})-1
\end{matrix}
\right)
\eea
In the proof of Lemma \ref{lemma: mainstream}, we proved that to satisfy
\bea
\det
\left(
\begin{matrix}
\displaystyle\exp(-\frac{m^2p^2}{l_1^2})-1, & \displaystyle\exp(-\frac{(mp+q)^2}{l_1^2})-\exp(-\frac{q^2}{l_1^2})\\
\displaystyle\exp(-\frac{m^2p^2}{l_2^2})-1, & \displaystyle\exp(-\frac{(mp+q)^2}{l_2^2})-\exp(-\frac{q^2}{l_2^2})
\end{matrix}
\right)
=0,
\eea
$l_1 = l_2$ should hold, which contradicts the assumption.
For $u(mp-q)-u(q)$, the proof goes exactly the same way.
Therefore, to satisfy 
\bea
G_{RBF}(\theta_1)+G_{period}(\eta_1) = G_{RBF}(\theta_2) + G_{period}(\eta_2),
\eea
$l_1 = l_2, s_1 = s_2$ should be held. When this holds, to satisfy
\bea
K_{RBF}(x_i, x_j;\theta_1)+K_{period}(x_i, x_j;\eta_1) = K_{RBF}(x_i, x_j;\theta_2) + K_{period}(x_i,x_j;\eta_2),
\eea
\begin{equation}\label{eqn: i=j}
    \sigma_1^2-\sigma_2^2 = \tau_2^2-\tau_1^2
\end{equation}
\begin{equation}\label{eqn: i>j}
(\sigma_1^2-\sigma_2^2)\exp(-\frac{(x_i-x_j)^2}{l_1^2}) = 
(\tau_2^2-\tau_1^2)\exp(-\frac{2\sin^2(\pi(x_i-x_j)/p)}{s_1^2})
\end{equation}
should hold. (\ref{eqn: i=j}) is obtained by setting $i=j$, and (\ref{eqn: i>j}) by $i>j$. Applying (\ref{eqn: i>j}) for $x_i, x_j$ which satisfy $x_i-x_j = mp$, $\sigma_1 = \sigma_2$ and $\tau_1 = \tau_2$ are obtained.

To sum up, if $G_{RBF}(\theta_1)+G_{period}(\eta_1) = G_{RBF}(\theta_2) + G_{period}(\eta_2)$ holds, $\theta_1= \theta_2$ and $\eta_1 = \eta_2$ should be satisfied, which completes the proof of Theorem \ref{thm:RBF-Periodic}. 

\subsection{Proof of Theorem \ref{thm:2-RBF}}

Since here, we consider observation points $\{x_i:\;1\le i \le n\}$ as $d$-dimensional vectors, and suppose following condition holds: 
\begin{equation}\label{eqn:2-RBF condition}
    X = \{\lVert x_i-x_j \rVert : 1\le i, j \le n\},\;|X| \ge 4
\end{equation}
Suppose that $G_{RBF}(\theta_1)$ is the kernel matrix of short-term RBF kernel and $G_{RBF}(\theta_2)$ long-term RBF kernel. In this model, denote $\theta_{ij} = (\sigma_{ij}, l_{ij})\;(i, j = 1, 2)$, where $j=1, 2$ represent parameters of short-term RBF kernel matrix, and long-term RBF kernel matrix each.

According to this notation, $l_{11}<l_{12},\; l_{21}<l_{22}$.
\begin{equation}\label{eqn:GGGG}
    G_{RBF}(\theta_{11}) + G_{RBF}(\theta_{12}) = G_{RBF}(\theta_{21}) + G_{RBF}(\theta_{22}) 
\end{equation}
means that for all $x_i, x_j$,
\begin{equation}\label{eqn:2-RBF equation}
    \sigma_{11}^2\exp(-\frac{(x_i-x_j)^2}{l_{11}^2}) +
    \sigma_{12}^2\exp(-\frac{(x_i-x_j)^2}{l_{12}^2}) =
    \sigma_{21}^2\exp(-\frac{(x_i-x_j)^2}{l_{21}^2}) +
    \sigma_{22}^2\exp(-\frac{(x_i-x_j)^2}{l_{22}^2}).
\end{equation}
It can be represented in vector form:
\bea
(\sigma_{11}^2, \sigma_{12}^2, -\sigma_{21}^2, -\sigma_{22}^2)
\left(
\begin{matrix}
\displaystyle\exp(-\frac{(x_i-x_j)^2}{l_{11}^2})\\
\displaystyle\exp(-\frac{(x_i-x_j)^2}{l_{12}^2})\\
\displaystyle\exp(-\frac{(x_i-x_j)^2}{l_{21}^2})\\
\displaystyle\exp(-\frac{(x_i-x_j)^2}{l_{22}^2})
\end{matrix}
\right)
=0.
\eea
Therefore, for
\bea
v(x) = 
\left(
\begin{matrix}
\displaystyle\exp(-\frac{x^2}{l_{11}^2})\\
\displaystyle\exp(-\frac{x^2}{l_{12}^2})\\
\displaystyle\exp(-\frac{x^2}{l_{21}^2})\\
\displaystyle\exp(-\frac{x^2}{l_{22}^2})
\end{matrix}
\right),\;
V= \{v(x)|x \in X\},
\eea
there exists nonzero vector which is orthogonal to all vectors in $V$, which means this vector is orthogonal to all vectors in $\mathcal{C}(V)$. Accordingly, rank $\mathcal{C}(V) \le 3$ should be held. We use the same logic as the proof of Theorem \ref{thm:RBF-Periodic}. 
We first show that if $l_{11}, l_{12}, l_{21}, l_{22}$ are mutually different, rank $\mathcal{C}(V) = 4$. Including the case of $l_{11} = l_{22}$ or $l_{21} = l_{22}$, when only one pair of $\{l_{11}, l_{12}, l_{21}, l_{22}\} $ is duplicated, we show that (\ref{eqn:2-RBF equation}) cannot be satisfied. 
Finally, if $l_{11} = l_{21}$ and $l_{12} = l_{22}$, we show that this implies $\sigma_{11} = \sigma_{21}$ and $\sigma_{12} = \sigma_{22}$, which means that the model is identifiable. 

For the proof, we need following lemmas.

\begin{lemma}\label{lemma: det not 0}
Let 
\bea
B(\zeta_1, \zeta_2, \zeta_3;a, b) = 
\left(
\begin{matrix}
\zeta_1-1 & \zeta_1^a-1 & \zeta_1^b-1\\
\zeta_2-1 & \zeta_2^a-1 & \zeta_2^b-1\\
\zeta_3-1 & \zeta_3^a-1 & \zeta_3^b-1
\end{matrix}
\right).
\eea
If $b>a>1,\; \zeta_1,\zeta_2,\zeta_3 > 1,$ and $ (\zeta_1-\zeta_2)(\zeta_2-\zeta_3)(\zeta_3-\zeta_1) \neq 0$ holds, $\det B \neq 0$.
\end{lemma}

\begin{proof}
Assume $\det B = 0$. Then there exists $(u_1, u_2, u_3) \neq (0, 0, 0)$ such that
\bea
u_1
\left(
\begin{matrix}
\zeta_1-1\\
\zeta_2-1\\
\zeta_3-1
\end{matrix}
\right)
+u_2
\left(
\begin{matrix}
\zeta_1^a-1\\
\zeta_2^a-1\\
\zeta_3^a-1
\end{matrix}
\right)
+u_3
\left(
\begin{matrix}
\zeta_1^b-1\\
\zeta_2^b-1\\
\zeta_3^b-1
\end{matrix}
\right)
=0.
\eea
When we define 
\begin{equation}\label{eqn:h}
    h(t) = u_1(t-1) + u_2(t^a-1) + u_3(t^b-1),
\end{equation}
$\zeta_1, \zeta_2, \zeta_3$ are mutually different roots of $h(t) = 0$. With $h(1) = 0$, $h(t)=0$ has 4 mutually different roots on $\{ t \ge 1 \}$.
By mean value theorem, 
\begin{equation}\label{eqn:h prime}
    \displaystyle h^{\prime}(t) = u_1 + au_2 t^{a-1}+ bu_3 t^{b-1} = 0
\end{equation}
has at least 3 different roots on $\{ t > 1\}$. By dividing into cases, we prove this cannot be the case. 

Case I. Suppose $u_1=0, u_2=0$. $u_3 \neq 0$ not to make $(u_1, u_2, u_3)$ zero vector. However, according to (\ref{eqn:h prime}), $h^{\prime}(t) = bu_3 t^{b-1}$, which means $h^{\prime}(t) \neq 0$ 0 on $\{t>=1\}$.
This also implies that $u_1 = 0$ and $u_3 = 0$ should not be satisfied simultaneously.

Case II. Suppose $u_1 = 0, u_2 \neq 0$ and $u_3 \neq 0$. According to (\ref{eqn:h prime}),
\bea
\displaystyle h^{\prime}(t) = au_2t^{a-1}(1+\frac{bu_3}{au_2}t^{b-a})
\eea
holds. Define 
\bea
\displaystyle P_1(t) = 1+\frac{bu_3}{au_2}t^{b-a}.
\eea
 Then $P_1(t)=0$ should have at least 3 roots on $\{t \ge 1\}$. This contradicts the fact that $P_1(t)$ is monotone function on $\{t \ge 1\}$, since 
\bea
\displaystyle P_1^{\prime}(t) = \frac{b(b-a)u_3}{au_2} t^{b-a-1}
\eea
does not change its sign on $\{t \ge 1\}$.

Case III. Suppose $u_1 \neq 0$. By dividing \ref{eqn:h prime} by $u_1$,
\bea
\displaystyle h^\prime (t) = u_1(1+\frac{au_2}{u_1}t^{a-1} + \frac{bu_3}{u_1}t^{b-1})
\eea
holds. Define 
\bea
\displaystyle P_2(t) = 1+\frac{au_2}{u_1}t^{a-1} + \frac{bu_3}{u_1}t^{b-1}.
\eea
Then $P_2(t) = 0$ should have at least 3 different roots on $\{t>1\}$. If $u_2=u_3=0$,
\bea
h(t) = u_1(t-1) + u_2(t^a-1) + u_3(t^b-1) = u_1(t-1)
\eea
cannot have 4 different roots on $\{t \ge 1\}$. Similarly, if $u_2 = 0$ or $u_3 = 0$, $P_2^\prime(t)$ cannot be 0 on $\{t>1\}$, which implies that $P_2$ is monotone, and cannot have 4 different roots on $\{t \ge 1\}$. 
Suppose $u_1u_2u_3 \neq 0$. Applying mean value theorem again, $P_2^{\prime}(t) = 0$ should have at least 2 different roots on $\{t>1\}$. Denote the smallest root and the second smallest root as $c_1, c_2$ respectively. Note
\bea
P_2^{\prime}(t) = \frac{a(a-1)u_2}{u_1}t^{a-2} +\frac{b(b-1)u_3}{u_1}t^{b-2},
\eea
which can be written as
\bea
P_2^{\prime}(t) = \frac{a(a-1)u_2}{u_1}t^{a-2}(1+\frac{b(b-1)u_3}{a(a-1)u_2}t^{b-a}).
\eea
Define 
\bea
P_3(t) = 1+\frac{b(b-1)u_3}{a(a-1)u_2}t^{b-a}.
\eea
To satisfy $P_2^{\prime}(c_1) = P_2^{\prime}(c_2)=0$, $P_3(c_1) = P_3(c_2) = 0$ should hold. However, $P_3(t)$ is a monotone function, which contradicts above. Therefore, the assumption is wrong, which concludes that $\det B \neq 0$.
\end{proof}

\begin{lemma}\label{lemma:mainstream_2}
Suppose that $l_{11}, l_{12}, l_{21}, l_{22}$ are mutually different. When we define 
\bea
v(x) = \left(
\begin{matrix}
\displaystyle\exp(-\frac{x^2}{l_{11}^2})\\
\displaystyle\exp(-\frac{x^2}{l_{12}^2})\\
\displaystyle\exp(-\frac{x^2}{l_{21}^2})\\
\displaystyle\exp(-\frac{x^2}{l_{22}^2})
\end{matrix}
\right),
\eea
and $V = \{v(x)|x \in X\}$, rank $\mathcal{C}(V) = 4$.
\end{lemma}

\begin{proof}
According to (\ref{eqn:2-RBF condition}), there exists
\bea
a_1, a_2, a_3 \in X \text{ such that }a_3>a_2>a_1>0.
\eea
Let
\bea
a = \frac{a_2}{a_1},\;b = \frac{a_3}{a_1},\;\alpha = \exp(-\frac{a_1^2}{l_{11}^2}),
\;\beta = \exp(-\frac{a_1^2}{l_{12}^2}), 
\;\gamma = \exp(-\frac{a_1^2}{l_{21}^2}), 
\;\delta = \exp(-\frac{a_1^2}{l_{22}^2}).
\eea
Since $l_{11}<l_{12}$ and $l_{21}<l_{22}$, without loss of generality, suppose
\bea
\min(l_{11}, l_{12}, l_{21}, l_{22}) = l_{11}.
\eea
For
\bea
v(0) = \left(
\begin{matrix}
1\\
1\\
1\\
1
\end{matrix}
\right), 
v(a_1) = \left(
\begin{matrix}
\alpha\\
\beta\\
\gamma\\
\delta
\end{matrix}
\right), 
v(a_2) = \left(
\begin{matrix}
\alpha^a\\
\beta^a\\
\gamma^a\\
\delta^a
\end{matrix}
\right), 
v(a_3) = \left(
\begin{matrix}
\alpha^b\\
\beta^b\\
\gamma^b\\
\delta^b
\end{matrix}
\right),
\eea
$\mathcal{C}(\{v(0), v(a_1), v(a_2), v(a_3)\}) \subset \mathcal{C}(V)$. Therefore, it is enough to show that 
\bea
\text{rank }\mathcal{C}(\{v(0), v(a_1), v(a_2), v(a_3)\}) = 4.
\eea
Define
\bea
A = \left(
\begin{matrix}
1&1&1&1\\
\alpha&\beta&\gamma&\delta\\
\alpha^a&\beta^a&\gamma^a&\delta^a\\
\alpha^b&\beta^b&\gamma^b&\delta^b
\end{matrix}
\right).
\eea
It is equivalent to show that $\det A \neq 0$.
Define 
\bea
\zeta_1 =\frac{\beta}{\alpha},\;\zeta_2 = \frac{\gamma}{\alpha},\; \zeta_3 = \frac{\delta}{\alpha}.
\eea
Since
\bea
\min(l_{11}, l_{12}, l_{21}, l_{22}) = l_{11},\; \min(\alpha, \beta, \gamma, \delta) = \alpha,
\eea
$\zeta_1, \zeta_2, \zeta_3 > 1$ holds. Also,
\bea
\det A = \alpha^{1+a+b} 
\begin{vmatrix}
1&1&1&1\\
1&\zeta_1&\zeta_2&\zeta_3\\
1&\zeta_1^a&\zeta_2^a&\zeta_3^a\\
1&\zeta_1^b&\zeta_2^b&\zeta_3^b
\end{vmatrix} = 
\begin{vmatrix}
1&1&1&1\\
0&\zeta_1-1&\zeta_2-1&\zeta_3-1\\
0&\zeta_1^a-1&\zeta_2^a-1&\zeta_3^a-1\\
0&\zeta_1^b-1&\zeta_2^b-1&\zeta_3^b-1
\end{vmatrix} = \det B
\eea
holds when we define
\bea
B = \left(
\begin{matrix}
\zeta_1-1&\zeta_2-1&\zeta_3-1\\
\zeta_1^a-1&\zeta_2^a-1&\zeta_3^a-1\\
\zeta_1^b-1&\zeta_2^b-1&\zeta_3^b-1
\end{matrix}
\right).
\eea
By Lemma \ref{lemma: det not 0}, $\det B \neq 0$, which means that $\det A \neq 0$. Therefore, rank $\mathcal{C}(V) = 4$.
\end{proof}
Consequently, to satisfy (\ref{eqn:GGGG}),
$l_{11}, l_{12}, l_{21}, l_{22}$ should not be mutually different.

\begin{lemma}\label{lemma: supplement_2RBF}
For mutually different real numbers $a, b, c>0$, define 
\bea
w(x) = \left(
\begin{matrix}
\displaystyle\exp(-\frac{x^2}{a^2})\\
\displaystyle\exp(-\frac{x^2}{b^2})\\
\displaystyle\exp(-\frac{x^2}{c^2})
\end{matrix}
\right).
\eea
Then for mutually different real numbers $x_1, x_2, x_3>0,\; \{w(x_1), w(x_2), w(x_3)\}$ is a linearly independent set.
\end{lemma}

\begin{proof}
Without loss of generality, suppose
\bea
0<x_1<x_2<x_3, a>b>c.
\eea
Define 
\bea
\alpha = \exp(-\frac{x_1^2}{a^2}), \beta = \exp(-\frac{x_1^2}{b^2}),   \gamma = \exp(-\frac{x_1^2}{c^2}), A = \frac{x_2^2}{x_1^2}, B = \frac{x_3^2}{x_1^2}.
\eea
Then
\bea
\alpha>\beta>\gamma, 1<A<B
\eea
holds. Define matrix C as
\bea
C = (w(x_1), w(x_2), w(x_3)) = 
\left(
\begin{matrix}
\alpha& \alpha^A& \alpha^B\\
\beta& \beta^A& \beta^B\\
\gamma& \gamma^A& \gamma^B\\
\end{matrix}
\right).
\eea
Then it is equivalent to show that $\det C \neq 0$. In addition,
\bea
\det C = 
\alpha\beta\gamma
\begin{vmatrix}
1&\alpha^{A-1}&\alpha^{B-1}\\
1&\beta^{A-1}&\beta^{B-1}\\
1&\gamma^{A-1}&\gamma^{B-1}
\end{vmatrix} &=& \alpha^{A+B-1}\beta\gamma
\begin{vmatrix}
1&1&1\\
1&(\displaystyle\frac{\beta}{\alpha})^{A-1}&(\displaystyle\frac{\beta}{\alpha})^{B-1}\\
1&(\displaystyle\frac{\gamma}{\alpha})^{A-1}&(\displaystyle\frac{\gamma}{\alpha})^{B-1}
\end{vmatrix}\\\\
&=&\alpha^{A+B-1}\beta\gamma
\begin{vmatrix}
1&1&1\\
0&(\displaystyle\frac{\beta}{\alpha})^{A-1}-1&(\displaystyle\frac{\beta}{\alpha})^{B-1}-1\\
0&(\displaystyle\frac{\gamma}{\alpha})^{A-1}-1&(\displaystyle\frac{\gamma}{\alpha})^{B-1}-1
\end{vmatrix}
\eea
holds. According to the following lemma, when we set $\displaystyle x_1 = \frac{\beta}{\alpha}$ and $\displaystyle x_2 = \frac{\gamma}{\alpha}$, 
$\det C = 0$ is equivalent to the fact that $x_1 = x_2$, which contradicts the assumption. This concludes the proof.
\end{proof}

\begin{lemma}\label{lemma:last}
for $0<x_1, x_2< 1$, if $x_1 \neq x_2$ and $1<A<B$, 
\bea
\begin{vmatrix}
x_1^{A-1}-1 & x_1^{B-1}-1 \\
x_2^{A-1}-1 & x_2^{B-1}-1
\end{vmatrix} \neq 0.
\eea
\end{lemma}
\begin{proof}
\bea
\begin{vmatrix}
x_1^{A-1}-1 & x_1^{B-1}-1 \\
x_2^{A-1}-1 & x_2^{B-1}-1
\end{vmatrix} =0
\eea
can be written as
\begin{equation}\label{eqn:expanded form}
x_1^{A-1}x_2^{B-1}-x_1^{B-1}x_2^{A-1}-x_1^{A-1}+x_1^{B-1}-x_2^{B-1}+x_2^{A-1} = 0.
\end{equation}
Since 
\bea
x_1^{A-1}-1 \neq 0,\; x_2^{A-1}-1 \neq 0,
\eea
(\ref{eqn:expanded form}) can be written as
\bea
\frac{x_1^{A-1}}{x_1^{A-1}-1}(x_1^{B-A}-1) = \frac{x_2^{A-1}}{x_2^{A-1}-1}(x_2^{B-A}-1).
\eea
When we define 
\bea
f(x) = \frac{x^B-x^A}{x^A-1},
\eea
$f(x_1) = f(x_2)$ holds. Since 
\bea
f^{\prime}(x) = \frac{x^{A-1}}{(x^A-1)^2}((B-A)x^B-Bx^{B-A}+A),
\eea
define 
\bea
g(x) = ((B-A)x^B-Bx^{B-A}+A).
\eea
Note
\bea
g^\prime(x) = (B-A)Bx^{B-A-1}(x^A-1).
\eea
Therefore, on $0<x<1$, $g(x)$ is monotonely decreasing, and $g(1) = 0$ holds, which implies that on  $0<x<1,\;g(x) > 0$. Namely, if $0<x<1$ and $1<A<B$, 
\bea
(B-A)x^B-Bx^{B-A} + A > 0 
\eea
stands. Therefore, since $f^\prime(x) > 0$ on $0<x<1$, $f$ is monotonely increasing. Consequently, for $0<x_1, x_2<1$, $f(x_1) = f(x_2)$ implies $x_1 = x_2$.
By Lemma \ref{lemma: supplement_2RBF}, for the case in which only one pair in $\{l_{ll}, l_{12}, l_{21}, l_{22}\}$ is duplicated, (\ref{eqn:GGGG})
cannot be satisfied.

Therefore, to satisfy (\ref{eqn:GGGG}), 
\bea
(l_{11}, l_{12}) = (l_{21}, l_{22}) 
\eea
should be satisfied. By substituting the value of $\lVert x_i-x_j \rVert$ in (\ref{eqn:2-RBF equation}),
$\sigma_{11} = \sigma_{21}$ and $\sigma_{12} = \sigma_{22}$ is obtained.
This concludes the proof of Theorem \ref{thm:2-RBF}.
\end{proof}
%
%
%%%%%%%%%%%%%%%%%%%%%%%%%%%%%%%%%%%%%%%%%%%%%%%%%%%%%%%%%%%%%%%%%%%%%%%%%%%%%%%%%%%%%%%%%%%%%%%%%%%%%%%%%%%%%%%%%%

\section{Numerical Examples}\label{sec:num_ex}

\subsection{RBF-Periodic mixed kernel GPR model}

\subsubsection{no \texorpdfstring{$x \in X$}{Lg} which satisfies \texorpdfstring{$\displaystyle\frac{x}{p} \notin \mathbb{N}$}{Lg}}

In this example, we set $p=7$. Suppose the set of observed covariates is $\{x_i:1\le i \le n\} = \{1, 8, 15, ...., 7n-6\},\; n \ge 3$. In this case, if $\sigma_1 = \sigma_2$, $l_1 = l_2,$ $\tau_1 = \tau_2$ and $s_1 \neq s_2$, then
\bea
G_{RBF}(\theta_1) + G_{period}(\eta_1) = G_{RBF}(\theta_2) + G_{period}(\eta_2) 
\eea
is always satisfied. As as specific example, consider
\bea
\{x_i:1\le i \le 6\} = \{1, 8, 15, 22, 29, 36\}.
\eea
If the mixed kernel is
\bea
\sigma_1=\sigma_2=l_1=l_2=\tau_1=\tau_2=1,
s_1 =1 \text{ and }s_2 =2,
\eea
then
\bea
G = G_{RBF}(\theta_1) + G_{period}(\eta_1) = G_{RBF}(\theta_2) + G_{period}(\eta_2) = 
\left(
\begin{matrix}
2&1&1&1&1&1\\
1&2&1&1&1&1\\
1&1&2&1&1&1\\
1&1&1&2&1&1\\
1&1&1&1&2&1\\
1&1&1&1&1&2\\
\end{matrix}
\right).
\eea
This is a non-identifiable case where $X$ does not contain any element $x$ which satisfies $\displaystyle \frac{x}{p} \notin \mathbb{N}$.

\subsubsection{no \texorpdfstring{$x \in X$}{Lg} which satisfies \texorpdfstring{$\displaystyle\frac{x}{p} \in \mathbb{N}$}{Lg}}
We consider the case that $\{x_i:1\le i \le n\}$ doesn’t contain any element $x$ such that $\displaystyle \frac{x}{p} \in \mathbb{N}$. Suppose that $p = 4$.
When $X = \{0, 1, 2, 3\}$, X doesn’t contain any element x which satisfies $\displaystyle \frac{x}{p} \in \mathbb{N}$.
For
\bea
\sigma_1=1.5720871, l_1=1.0045602, \tau_1 = 1.4284245, s_1 = 1.2011224\\
\sigma_2=1.2295748, l_2=1.4468554, \tau_2 = 1.7320508, s_2 = 0.9540646
\eea
$G_{RBF}(\theta_{1}) + G_{period}(\eta_{1}) = G_{RBF}(\theta_{2}) + G_{period}(\eta_{2})$ holds. Define
\bea
v(x) = 
\left(
\begin{matrix}
\displaystyle\exp(-\frac{x^2}{l_1^2})\\
\displaystyle\exp(-\frac{x^2}{l_2^2})\\
\displaystyle\exp(-\frac{\cos(2\pi x/p)}{s_1^2})\\
\displaystyle\exp(-\frac{\cos(2\pi x/p)}{s_2^2})
\end{matrix}
\right).
\eea
The parameters above are obtained by solving the relation below analytically: 
\bea
v(0)-3v(1)+6v(2)-2v(3) = (0, 0, 0, 0).
\eea
Note in this case,
\bea
G=G_{RBF}(\theta_1)+G_{period}(\eta_1) =
G_{RBF}(\theta_2)+G_{period}(\eta_2) =
\eea
\bea
\left(
\begin{matrix}
4.5118543 & 1.9376702 & 0.5570357 & 1.0205292\\
1.9376702 & 4.5118543 & 1.9376702 & 0.5570357\\
0.5570357 & 1.9376702 & 4.5118543 & 1.9376702\\
1.0205292 & 0.5570357 & 1.9376702 & 4.5118543
\end{matrix}
\right).
\eea

\subsection{2-RBF mixed kernel GPR model}
For 2-RBF mixed kernel GPR model in 3-dimensional space, observations on 6 different points can even be insufficient for the identifiability. For example, assign 6 points coordinate as $(\pm1, \pm1, 0), (0, 0, \pm \sqrt2)$, and parameters as follows:
\bea
\sigma_{11} = 24, \sigma_{12} = 32\sqrt{15}, l_{11}=\frac{1}{\sqrt{\log4}}, l_{12} = \frac{1}{\sqrt{\log16}}\\
\sigma_{21} = 81, \sigma_{22} = 25\sqrt{15}, l_{11}=\frac{1}{\sqrt{\log9}}, l_{12} = \frac{1}{\sqrt{\log25}}
\eea
Then, 6 $\times$ 6 matrix $G_{RBF}(\theta_{11}) + G_{RBF}(\theta_{12})$ and $G_{RBF}(\theta_{21}) + G_{RBF}(\theta_{22})$ are identical and
\bea
G = G_{RBF}(\theta_{11}) + G_{RBF}(\theta_{12}) = G_{RBF}(\theta_{21}) + G_{RBF}(\theta_{22})=
\eea
\bea
\left(
\begin{matrix}
15936 & 1104 & 1104 & 96 & 1104 & 1104 \\
1104 & 15936 & 96 & 1104 & 1104 & 1104 \\
1104 & 96 & 15936 & 1104 & 1104 & 1104 \\
96 & 1104 & 1104 & 15936 & 1104 & 1104 \\
1104 & 1104 & 1104 & 1104 & 15936 & 96 \\
1104 & 1104 & 1104 & 1104 & 96 & 15936 \\
\end{matrix}
\right).
\eea
In terms of model interpretation, the model explained by $\theta_{11}$ and $\theta_{12}$ poses an emphasis on long term smoothness, and the model by $\theta_{21}$ and $\theta_{22}$ on the short term smoothness, relatively.
This non-identifiable situation occurs because $|X| = 3$.

\section{Discussion}\label{sec:disc}

In this paper, we define the identifiability of mixed kernel GPR models. 
Then we prove the theorems regarding the identifiability of two simple mixed GPR models, which are RBF-Periodic mixed GPR model and 2-RBF mixed GPR model. 
After analytic proofs, we give some numerical examples about non-identifiable cases regarding the models. Non-identifiability is derived from lack of satisfaction of corresponding conditions.   

\section*{Acknowledgements}

Jaeyong Lee was supported by the National Research Foundation of Korea (NRF) grants funded by the Korean government (MSIT) (No. 2018R1A2A3074973 and 2020R1A4A1018207).

%% ====================================================================


\begin{references}

\bibitem[Atienza et al.(2006)]{atienza2006new} Atienza, N., Garcia-Heras, J., Munoz-Pichardo, J. (2006). A new condition for identifiability
of finite mixture distributions, \textit{Metrika}, \textbf{63}, 215–221.

\bibitem[Barndorff-Nielsen(1965)]{barndorff1965identifiability} Barndorff-Nielsen, O. (1965). Identifiability of mixtures of exponential families, \textit{Journal of
Mathematical Analysis and Applications} \textbf{12}, 115–121.

\bibitem[Bousquet et al.(2011)]{bousquet2011advanced} Bousquet, O., von Luxburg, U., Rätsch, G. (2011). \textit{Advanced lectures on machine learning:
ML Summer Schools 2003, Canberra, Australia, February 2-14, 2003, Tübingen,
Germany, August 4-16, 2003, Revised Lectures}, \textbf{3176}, Springer.


\bibitem[Choi and Lee(2017)]{choi2017poisson} Choi, J., Lee, K. (2017). Poisson linear mixed models with arma random effects covariance
matrix, \textit{Journal of the Korean Data and Information Science Society}, \textbf{28}, 927–936.

\bibitem[Daemi et al.(2019)]{daemi2019identification} Daemi, A., Alipouri, Y., Huang, B. (2019). Identification of robust gaussian process regression
with noisy input using em algorithm. \textit{Chemometrics and Intelligent Laboratory Systems}, \textbf{191}, 1–11.
\bibitem[Gelman et al.(2013)]{gelman2013bayesian} Gelman, A., Carlin, J.B., Stern, H.S., Dunson, D.B., Vehtari, A., Rubin, D.B. (2013).
\textit{Bayesian data analysis}, CRC press.

\bibitem[Huang et al.(2014)]{huang2014estimating} Huang, M., Li, R., Wang, H., Yao, W. (2014). Estimating mixture of gaussian processes by
kernel smoothing, \textit{Journal of Business \& Economic Statistics}, \textbf{32}, 259–270.


\bibitem[Kostantinos(2000)]{kostantinos2000gaussian} Kostantinos, N. (2000). Gaussian mixtures and their applications to signal processing. \textit{Advanced
signal processing handbook: Theory and implementation for radar, sonar, and
medical imaging real time systems}, 1-35, CRC Press.

\bibitem[Li et al.(1963)]{li2020hybrid} Li, Z., Guo, F., Chen, L., Hao, K., Huang, B. (2020). Hybrid kernel approach to gaussian
process modeling with colored noises, \textit{Computers \& Chemical Engineering}, \textbf{143}, 107067.

\bibitem[Noh et al.(2018)]{noh2018analysis} Noh, M., Ok, Y.J., Na, M.H., Ng, C.T. (2018). Analysis of degradation data using double
hierarchical generalized linear model, \textit{The Korean Data \& Information Science Society}, \textbf{29}, 217–228.

\bibitem[Teicher(1963)]{teicher1963identifiability} Teicher, H. (1963). Identifiability of finite mixtures. \textit{The annals of Mathematical statistics},
1265–1269.

\bibitem[Wilson and Adams(2013)]{wilson2013gaussian} Wilson, A., Adams, R. (2013). Gaussian process kernels for pattern discovery and extrapolation. In \textit{International conference on machine learning}, 1067–1075, PMLR.

\bibitem[Yakowitz and Spragins(1968)]{yakowitz1968identifiability} Yakowitz, S.J., Spragins, J.D. (1968). On the identifiability of finite mixtures. \textit{The Annals of Mathematical Statistics}, 209–214.

\end{references}
\end{document}